\documentclass[12pt]{amsart}

\usepackage{mymathmacros}
\usepackage{formatmacros}
\usepackage{analysismacros}
\usepackage{alphamacros}
\usepackage{linalgmacros}
\usepackage{refmacros}

\usepackage{pst-plot, pst-math, pst-func}
\usepackage{pstricks-add}

\begin{document}

\title[Anosov Structures in the Plane]{Foliations and Conjugacy: \\ Anosov Structures in the Plane}
\author{Jorge Groisman}
	\address{Instituto de Matem\'atica y Estad\'{\i}stica ÒProf. Ing. Rafael LaguardiaÓ, 
	Facultad de Ingenier\'{\i}a Julio Herrera y Reissig 565 11300, MONTEVIDEO, Uruguay}
	\email{jorge.groisman@gmail.com}
\author{Zbigniew Nitecki}
	\address{Department of Mathematics, Tufts University, Medford, MA 02155}
	\email{zbigniew.nitecki@tufts.edu}
\thanks{Both authors thank Lluis Alseda and the Department of Mathematics at the Universitat Autonoma de Barcelona 
for  
hospitality and support during the summer of 2012, when this project was started.
The first author thanks the Mathematics Department at Tufts University 
for its hospitality and support during a visit in May 2013, and IMERL for its support.
The second author thanks IMERL for its hospitality and support during a visit in March 2013, 
and the Faculty Research Awards Committee at Tufts for a grant-in-aid that helped support his travel to Montevideo 
during that visit.
}

\begin{abstract}
In a non-compact setting, the notion of hyperbolicity, 
and the associated structure of stable and unstable manifolds (for unbounded orbits),
is highly dependent on the choice of metric used to define it.  We consider the simplest version of this,
the analogue for the plane of Anosov diffeomorphisms, studied earlier by W. White and P. Mendes.
The two known topological conjugacy classes of such diffeomorphisms are linear hyperbolic automorphisms
and translations.  We show that if the structure of stable and unstable manifolds is required to be preserved
by these conjugacies, the number of distinct equivalence classes of Anosov diffeomorphisms 
in the plane becomes infinite.
\end{abstract}
\newcommand{\muprod}[2]{\ensuremath{\langle #1, #2 \rangle_{\mu}}}
\newcommand{\distmu}[2]{\dists{\mu}{#1}{#2}}
\newcommand{\elsmuof}[1]{\elsof{\mu}{#1}}
\newcommand{\Foliation}{\ensuremath{\mathcal{F}}}
\newcommand{\Fsup}[1]{\ensuremath{{\Foliation}^{#1}}}
\newcommand{\Ffs}{\ensuremath{\Fsup{s}}}
\newcommand{\Ffu}{\ensuremath{\Fsup{u}}}
\newcommand{\metric}{\ensuremath{\mu}}
\newcommand{\mlen}[1]{\ensuremath{\norm{#1}_{\mu}}}
\newcommand{\fnof}[1]{\ftoof{n}{#1}}
\newcommand{\angleprod}[2]{\ensuremath{\langle #1, #2 \rangle}}
\renewcommand{\interior}[1]{\ensuremath{\mathrm{int}\ #1}}
\newcommand{\calA}{\ensuremath{\mathcal{A}}}
\newcommand{\calAs}[1]{\ensuremath{\calA_{#1}}}
\newcommand{\VC}{\ensuremath{V_{\Curve}}}
\newcommand{\calVC}{\ensuremath{\calV_{\Curve}}}
\newcommand{\access}[2]{\ensuremath{{\mathcal{N}\left(#1,#2\right)}}}
\newcommand{\prolong}{\mathcal{J}}
\newcommand{\prolim}[2]{\ensuremath{\prolong_{#1}\left(#2\right)}}
\newcommand{\vstable}{\ves{e}{s}}
\newcommand{\vunstable}{\ves{e}{u}}

\maketitle


\section{Introduction}\label{sec:intro}
Diffeomorphisms on compact manifolds satisfying a global hyperbolicity condition, or \emph{Anosov diffeomorphisms},
have been very extensively studied in the past fifty years.  The hyperbolicity condition implies the existence of a transverse
pair of foliations by stable and unstable manifolds, with serious dynamic consequences 
(transitivity, density of periodic points, etc.).  The analogous condition in a non-compact setting does not in general 
imply such consequences.  A striking illustration of this difference is Warren White's construction \cite{White} 
of a complete Riemannian metric on the plane \Realstwo{} for which the translation $(x,y)\mapsto(x+2,y)$ 
is hyperbolic, although in every possible sense there is no recurrence at all.  White's example prompted Pedro Mendes
\cite{Mendes} to ask whether, at least in dimension 2, White's example together with the obvious example of a linear
hyperbolic automorphism of \Realstwo{} gives all possible Anosov diffeomorphisms (of \Realstwo) 
up to topological conjugacy.  This paper reports on an unsuccessful attempt to answer Mendes' question.
In the process of studying this problem, we formulate a stronger equivalence relation between Anosov diffeomorphisms
which takes into account the structure of the stable and unstable foliations, and find a wealth of examples 
of Anosov diffeomorphisms of \Realstwo{} which are not equivalent in this sense.

In a compact setting, the existence of a splitting in the tangent bundle implies the existence of stable 
and unstable foliations; Mendes asks whether this is also true in the setting of \Realstwo.  We have not attempted
to answer this question, as standard techniques for proving stable manifold theorems involve uniform estimates on 
the deviation between a diffeomorphism and its linearization at a point, something which is easily established 
in a compact setting but not in general in the plane.  Instead, we take as our starting definition the existence of 
stable and unstable foliations, in keeping with the definition adopted by Mendes in \cite{Mendes}.

\begin{definition}\label{dfn:Anosov}
	An \deffont{Anosov structure} on \Realstwo{} for a  diffeomorphism \selfmap{f}{\Realstwo} consists of  
	a complete Riemannian metric \metric{} on \Realstwo{} and 
	\begin{description}
		\item[Stable and Unstable Foliations] two continuous foliations \Ffs{} and \Ffu{} with \Cr{1} leaves
		respected by $f$: the image of a leaf of \Ffs{} \resp{\Ffu} is again a leaf of \Ffs{} \resp{\Ffu};
		\item[Hyperbolicity] there exist constants $C$ and  $\lam>1$ such that for any positive integer $n$
		and any vector \vv{} tangent to a
		leaf of \Ffu{}.
		\begin{equation*}
			\mlen{Df^{n}(\vv)}\geq C\lam^{n}\mlen{\vv}
		\end{equation*}
		while for any vector \vv{} tangent to a leaf of \Ffs{}
		\begin{equation*}
			\mlen{Df^{n}(\vv)}\leq C\lam^{-n}\mlen{\vv}
		\end{equation*}
		where \mlen{\vv} denotes the length of a vector using the metric \metric.
	\end{description}
\end{definition}
 We shall use the adjectives \emph{Anosov}, \emph{stable} and \emph{unstable} in the natural way:  
 a diffeomorphism is \deffont{Anosov} if it has an  Anosov structure; the leaf of \Ffs{} \resp{\Ffu} through a point is
 its \deffont{stable} \resp{\deffont{unstable}} \deffont{leaf}.  
 
 On a compact manifold, all metrics are uniformly equivalent, which means
 that if there exists an Anosov structure for $f$, the same foliations  together with any other metric will, 
 with an adjustment of the
 constants $C$ and \lam, also form an Anosov structure for $f$. Furthermore, the stable leaf through
 a point $x$ is its stable manifold, in the sense that it consists of all points $y$ for which the \metric-distance between
\fnof{x} and \fnof{y} converges to zero as $n\to\infty$ (with the analogous property 
with respect to ``backward time'' for points on the 
unstable leaf).  In particular, the foliation can be recognized in terms of the topological dynamics of the system.
This all disappears when we move to noncompact settings:  by switching to a metric with a different uniform structure, 
we can find Anosov structures  using different foliations, and other metrics
which do not support any Anosov structure (see \refer{thm}{equiv}). 
We stress also that our metric is assumed to be complete to avoid cheap pathologies.

While Mendes posed his question in terms of topological conjugacy, it seems more appropriate to regard the foliation
as part of the structure of an Anosov diffeomorphism.  Accordingly, we propose to study the following strengthening 
of topological conjugacy:

\begin{definition}\label{dfn:conj}
	Two Anosov structures of the plane, with respective Anosov diffeomorphisms \Rtransf{f}{2} and \Rtransf{g}{2},
	are \deffont{equivalent} if there exists a homeomorphism \Rtransf{h}{2} conjugating
	$f$ and $g$ ($h\compose f=g\compose h$)  which takes the stable \resp{unstable} foliation of $f$ 
	to the stable \resp{unstable} foliation of $g$.  The homeomorphism $h$ will be referred to as 
	a \deffont{foliated conjugacy} between $f$ and $g$.
\end{definition}

Our focus in this paper is on identifying a variety of equivalence classes of Anosov structures on the
plane.  Mendes showed in \cite{Mendes} that an Anosov diffeomorphism of the
plane has at most one non-wandering point; in particular, 
a linear hyperbolic automorphism of \Realstwo{} has the origin as
its unique non-wandering point, while a translation has none.  Our examples are all topologically conjugate to one
of these specific examples but as we shall see they represent an infinite family of Anosov structures 
with no foliated conjugacies between them.

\section{New Examples of Anosov structures}\label{sec:examples}

Recall that as a consequence of the Riemann mapping theorem, any open subset of \Realstwo{} which is homeomorphic 
to a disc (say, to the open unit disc) is actually diffeomorphic to all of \Realstwo.  We will refer to any such subset
of \Realstwo{} as an \deffont{open disc} in \Realstwo.
Our strategy for creating new examples
will be to consider open discs which are mapped onto themselves by the linear hyperbolic map
\begin{equation*}
	\Tof{x,y}=(2x, y/2).
\end{equation*}
Note that any hyperbolic linear automorphism of \Realstwo{} is linearly conjugate to this example, so picking this particular
automorphism presents no loss of generality.  If \gof{x,y} is a positive real function which is constant along orbits of 
$T$, then a new metric can be defined by scaling the tangent space at each point by this function: the invariance of
$g$ insures that in this metric (just as in the Euclidean one) 
all horizontal \resp{vertical} vectors are stretched by a factor of 2 \resp{shrunk by a factor of \half}.
If we can define the scaling function on an invariant open disc in a way that renders the resulting metric complete, then 
this metric together with the horizontal \resp{vertical} foliations (intersected with our disc) 
defines an Anosov 
structure for the restriction of $T$ to this disc, and 
any diffeomorphism from the disc onto \Realstwo{} conjugates $T$ with a transformation of the plane, with
an Anosov structure given by the image of the horizontal and vertical foliations of the disc.

We start by constructing two basic examples, one containing the origin, the other not containing the origin.

Note that the function 
\begin{equation*}
	\tauof{x,y}=xy
\end{equation*}
is invariant under the linear transformation
\begin{equation*}
	T:(x,y)\mapsto\left(2x,\frac{y}{2}\right).
\end{equation*}

\textbf{First example (not containing origin):} Consider the open set
\begin{equation}\label{eqn:band}
	\calU\eqdef\setbld{(x,y)\in\Realstwo}{x>0,\quad \recip{x}<y<\frac{2}{x}}.
\end{equation}
This is an open disc, invariant under $T$, and containing no fixed points of $T$.  
Define a Riemann metric on \calU{} by setting the new inner product of two vectors \vv{} and \vw{} at (x,y) to be
\begin{equation}\label{eqn:metric}
	\angleprod{\vv}{\vw}=(\gof{x,y})^{2}(\vv\cdot\vw)
\end{equation}
where 
\begin{equation*}
	\gof{x,y}=\recip{xy-1}+\recip{2-xy}=\recip{(xy-1)(2-xy)}
\end{equation*}
and the dot product is the usual (Euclidean) inner product.
Of course, $g$ is the composition with $\tau$ of the function \vphiof{t} defined on \opint{1}{2} by
\begin{equation*}
	\vphiof{t}=\recip{(t-1)(2-t)}
\end{equation*}
and as such is $T$-invariant.  Furthermore,
\begin{itemize}
	\item \vphiof{t} is unimodal, with a minimum value of $\vphiof{\frac{3}{2}}=4$ 
	and diverging monotonically to $+\infty$ at the ends
	of the interval;
	\item Any improper integral involving the endpoints diverges to $+\infty$:
	\begin{equation*}
		\int_{1}^{\frac{3}{2}}\vphiof{t}\dt=\int_{\frac{3}{2}}^{2}\vphiof{t}\dt=+\infty.
	\end{equation*}
\end{itemize}

\begin{lemma}\label{lem:complete}
	The metric on the open set \calU{} (\refer{eqn}{band}) defined by \refer{eqn}{metric} is complete.
\end{lemma}

\begin{proof}
To this end, we note first that since the (Euclidean) length of every vector is multiplied by at least $\vphiof{\frac{3}{2}}=4$,  
a sequence of points
in \calU{} which is Cauchy in the new metric is also Cauchy in the Euclidean metric, and hence converges in \Realstwo{}
to a point of the closure of \calU.  If it converges to a point of \calU, then since $g$ is locally bounded (near the limit point),
it converges there in the new metric.  It remains to show that no sequence which is Cauchy in the new metric can converge
in \Realstwo{} to a point on the boundary of \calU.  We prove this by contradiction.

Suppose $\ps{i}\in\calU$ converge to a point $q=(\xso,\yso)\in\bdry\calU$.  
Note that $\xso>0$ and $\tauof{\xso,\yso}=1\text{ or }2$.

Pick $a>0$ so that \xso{} is between $a$ and $2a$, and consider the 
compact set
\begin{equation}\label{eqn:fund}
	\mDs{a}\eqdef\setbld{(x,y)}{a\leq x\leq 2a,\quad\recip{x}\leq y\leq\frac{2}{x}}
\end{equation}
which is the intersection of the closed vertical ``band'' $\clint{a}{2a}\times\Reals$ with the closure of \calU.
Since $\ps{i}\to q$, eventually these points all lie in \mDs{a}.  
Now, the map $(x,y)\mapsto(x,\tauof{x,y})$ is a \Cr{\infty} diffeomorphism taking \mDs{a}{} onto the 
rectangle \rect{a}{2a}{1}{2}.  
By compactness of \mDs{a}, it is bi-Lipschitz, so for some positive constant $C$, the (Euclidean) 
distance between  points in \mDs{a}{} is bounded below by $C$ times the (Euclidean) distance 
between the corresponding points in the rectangle, 
which is  bounded below by the difference between their $\tau$-values.  

If $z,\zp\in\interior{\mDs{a}}$, let $\tauof{z}=\alpha$ and $\tauof{\zp}=\beta=\alpha+D$. 
Then for any curve \gam{}  from $z$ to \zp{} in \interior{\mDs{a}} parametrized by (Euclidean) arc length, we can pick points
\qs{i}, \rs{i}, $i=1,..,n$ in \gam{}  so that 
\begin{itemize}
	\item $\tauof{\qs{i}}=\alpha+\frac{i-1}{n}D$
	\item $\tauof{\rs{i}}=\alpha+\frac{i}{n}D$
	\item Along the segment \gams{i} of \gam{} from \qs{i} to \rs{i}, $\tau$ is always between the two endpoint values.
\end{itemize} 
Then the Euclidean length of \gams{i} is bounded below by $C\delof{t}$ where $\delof{t}=\frac{D}{n}$, and since its velocity vector is multiplied by
at least $\vphiof{ \tauof{\qs{i}}}=\vphiof{\alpha+\frac{i-1}{n}D}$ in measuring the new length of \gams{i}, we see that 
the new length of \gams{i} is bounded below by  $C\vphiof{\alpha+(i-1)\delof{t}}\delof{t}$, and the new length of \gam{} is
bounded below by $C\sum_{i=1}^{n}\vphiof{\alpha+(i-1)\delof{t}}\delof{t}$, which is the lower sum for
\begin{equation*}
	C\int_{\alpha}^{\beta}\vphiof{t}\dt.
\end{equation*}

Applying this to a subsequence of \ps{i} for which $\tau$ is strictly increasing, we see that the new distance from \ps{1} to
\ps{j} goes to infinity as $j\to\infty$, contradicting the assumption that the \ps{i} were Cauchy in the new metric.
\end{proof}

\textbf{Second Example (containing the origin):} As a disc containing the origin, we take
\begin{equation}\label{eqn:hyp}
	\calV\eqdef\setbld{(x,y)\in\Realstwo}{\abs{\tauof{x,y}}<1}.
\end{equation}
This is an open disc bounded by the two hyperbolas $\tauof{x,y}=\pm1$.  
The construction of the new metric is completely analogous to the previous case: we use \refer{eqn}{metric} but with
the defining function \vphiof{t} changed:  we now take
\begin{equation}\label{eqn:newphi}
	\vphiof{t}\eqdef\recip{1-t}+\recip{t+1}=\frac{2}{1-t^{2}}.
\end{equation}
The function \gof{x,y} takes its minimum value $g=2$ along the coordinate axes, 
and goes to infinity at the boundary of \calV.

To show the analogue of \refer{lem}{complete}, we again note that since $\gof{x,y}\geq2$ for any point of \calV{},
any sequence which is Cauchy in the new metric is also Cauchy in the Euclidean metric, and hence converges to
a point in the closure of \calV{}.  We need to show that a sequence which converges to a point on one of the two curves 
$\tauof{x,y}=\pm1$ cannot be Cauchy in the new metric.  We sketch the proof if a sequence converges
to a point on the boundary in the first quadrant: $q=(\xso,\yso)$ with $\xso>0$ and $\tauof{\xso,\yso}=1$:
we replace the fundamental domain \mDs{a} defined in \refer{eqn}{fund} by
\begin{equation*}
	\tmDs{a}\eqdef\setbld{(x,y)}{a\leq x\leq 2a,\quad-1\leq \tauof{x,y}\leq1}
\end{equation*}
and repeat the argument for \refer{lem}{complete}.

In both of these constructions, the fact that $\gof{\Tof{x,y}}=\gof{x,y}$ means that for any vector \vv{} at a point $(x,y)$, the
ratio of lengths between it and its image is the same for the new metric as the old one, and so the new metric 
(together with the horizontal and vertical foliations) provides
an Anosov structure for the restriction of $T$ to \calU{} \resp{to \calV}.  But \calU{} \resp{\calV} is an open disc, and 
hence by the Riemann mapping theorem there is a diffeomorphism of \calU{} \resp{\calV} onto the whole plane;  this
conjugates $T$ with some diffeomorphism $F$ of \Realstwo{} onto itself, and the push-forward of the new metric from our
subset to \Realstwo{} together with the images of the horizontal and vertical foliations of that set, 
yields an Anosov structure for $F$.  
Note in particular that the first case of our construction (\calU{} excludes the origin) provides an alternate construction
of a fixed point-free diffeomorphism of the plane which is Anosov.  

In fact, by their nature, these two constructions do not give counterexamples to Mendes' original conjecture: 
the fixed point-free construction yields a diffeomorphism which is conjugate to a translation and the one containing the
origin yields a conjugate of a linear hyperbolic transformation.  However, neither example has a \emph{foliated} conjugacy with the standard examples on the whole plane.
We formulate this as

\begin{theorem}\label{thm:equiv}
	\begin{enumerate}
		\item There exists an Anosov structure on the plane whose underlying diffeomorphism
		is a linear hyperbolic automorphism of \Realstwo{}, but for which
		the stable and unstable foliations cannot both be mapped to the standard foliations (by
		horizontal and vertical lines) for the linear hyperbolic map.
		
		\item There exists an Anosov structure on the plane whose underlying diffeomorphism is 
		topologically conjugate to the translation $(x,y_)\to(x+1,y)$ on \Realstwo, but whose stable
		foliation is not homeomorphic the the stable foliation in White's example.
	\end{enumerate}
\end{theorem}

	\begin{proofof}{1}
		We refer to the second example.  Let \map{\psi}{\opint{-1}{1}}{\opint{-\infty}{\infty}} be a 
		strictly increasing continuous function which equals the identity on \opint{-\half}{\half}
		such that $\psiof{t}\to\pm\infty$ as $t\to\pm\infty$.  Then
		\begin{equation*}
			\hof{x,y}=(x,\psiof{xy}\abs{y})
		\end{equation*}
		is a homeomorphism of \calV{} onto \Realstwo{}, and
		\begin{multline*}
			\hof{\Tof{x,y}}=\hof{2x,y/2}=(2x,\psiof{xy}\abs{y}/2)\\=\Tof{x,\psiof{xy}\abs{y}}=\Tof{\hof{x,y}}
		\end{multline*}
		so it conjugates the linear hyperbolic automorphism $T$ with itself.
		
		However, the corresponding Anosov structures are not equivalent.  
		The stable \resp{unstable} leaves of the example in \calV{} are 
		the intersections with \calV{} of vertical \resp{horizontal} lines, and it is clear that the stable \resp{unstable} 
		leaf through a point $(\xso,\yso)$ off the coordinate axes does not intersect the unstable \resp{stable} leaf 
		through any point with $\abs{x}>\abs{1/y}$ \resp{$\abs{y}\abs{1/x}$}.  
		But in the standard Anosov structure for $T$, the stable \resp{unstable} leaf through a point is the 
		vertical \resp{horizontal} line through that point, and so \emph{every} stable leaf intersects \emph{every}
		unstable leaf in this structure. Since intersection of leaves is an invariant of homeomorphism, the two
		Anosov structures are not equivalent.
   
	\end{proofof}
	
	\begin{proofof}{2}
		In the fixed-point free case, we invoke the celebrated \emph{Translation Theorem} of Brouwer \cite{Brouwer}. 
		He showed that given a fixed-point free
		orientation-preserving homeomorphism of the plane, through every point there is an 
		embedded line which separates its pre image from its image; this is sometimes called a
		\deffont{Brouwer line}.  The region bounded by a Brouwer line and its image is a kind of 
		fundamental domain: its images (in forward and backward time) fill out an invariant open disc 
		for which the restriction of the homeomorphism is topologically conjugate to the horizontal
		translation $(x,y)\to(x+1,y)$.  Brouwer lines are clearly taken to Brouwer lines by any 
		conjugacy.
		
		In the first example, any vertical line intersects \calU{} in a Brouwer line for the restriction of the
		transformation $T$ to \calU, and the invariant
		open disc it induces is clearly all of \calU.  These intersections are the stable leaves of the Anosov
		structure constructed in that example.  
		
		By contrast, if we refer to Figure 1 in White's paper \cite[p. 669]{White} we see that 
		(even though \emph{some}  stable leaves are Brouwer lines)
		there are others (the parabola-like ones) which are not Brouwer lines.%
		\footnote{A similar phenomenon occurs in our last example, in \refer{sec}{parallel} (see \refer{fig}{frame}), 
		as well as the example in \refer{sec}{access} illustrated by \refer{fig}{stmD}. }
		However,  the underlying transformation is by construction a parallel translation.
	\end{proofof}


\section{Accessibility: More Examples}\label{sec:access}

An invariant of our equivalence relation between Anosov structures is the structure of \emph{accessibility conditions}, 
generalizing the observation which distinguished our second example from linear hyperbolic automorphisms of the
whole plane.  Given the pair of transverse foliations \Ffs{} and \Ffu{} coming from an Anosov structure, we say $q$ is
\deffont{$\boldsymbol{n}$-accessible} from $p$ if there is a path from $p$ to $q$ consisting of arcs of leaves in 
\Ffs{} and \Ffu{}--that is, there is a finite sequence of points 
\begin{equation*}
	p=\ps{0},\ps{1},\dots,\ps{n}=q
\end{equation*}
such that for $i=1,\dots,n-1$, \ps{i} and \ps{i+1} lie in the same stable or unstable leaf.

\begin{remark}\label{rmk:access}
	Given an Anosov structure on \Realstwo{}, for every pair of points $p,q\in\Realstwo$ there exists $n=n(p,q)$
	such that $q$ is $n$-accessible from $p$.
\end{remark}

To see this, fix $p$ and let \calAs{n} be the set of points which are $n$-accessible from $p$.  Since $n$-accessibility 
implies $k$-accessibility for every $k>n$,  these form a nested, increasing family of subsets of \Realstwo
\begin{equation*}
	\calAs{n}\subset\calAs{n+1}.
\end{equation*}
Also, if a point is $n$-accessible from $p$, then by the local product structure (\ie{} transversality) of \Ffs{} and
\Ffu{}, its stable \resp{unstable} leaf intersects the unstable \resp{stable} leaf of every point in a (product) 
neighborhood.  Thus
\begin{equation*}
	\calAs{n}\subset\interior\calAs{n+1}
\end{equation*}
from which it follows that 
\begin{equation*}
	\calAs{\infty}\eqdef\bigcup_{n\in\Integers}\calAs{n}
\end{equation*}
is open.
However, \calAs{\infty} is also closed: if $\qs{i}\in\calAs{\ns{i}}$ converge to $q$, 
then since every point has a (product) neighborhood
$U$ of points from which it is $2$-accessible, once  we have $\qs{i}\in U$, we also have
$q\in\calAs{\ns{i}+2}\subset\calAs{\infty}$.
Thus by connectedness of \Realstwo, $\calAs{\infty}=\Realstwo$.

If we define, for each pair of points $p,q\in\Realstwo$
\begin{equation*}
	\access{p}{q}=\min\setbld{n}{q \text{ is $n$-accessible from }p}
\end{equation*}
then this number can vary with the pair of points, but its supremum over all pairs of points, which we can call the
\deffont{degree of inaccessibility} of the Anosov structure, is an invariant of foliated conjugacy.

\begin{theorem}\label{thm:access}
	There exist Anosov structures with arbitrarily high finite degree of inaccessibility.  These can be chosen
	to be fixed-point free or to have a fixedpoint.
\end{theorem}

\begin{proof}
	We construct examples by modifying the examples of the previous section.  We will work with the first
	(fixedpoint-free) example for definiteness, but it will be clear how to make an analogous change in the 
	second example (with a fixed point).	
	Recall the set 
	\begin{equation*}
		\mDs{1}=\setbld{(x,y)}{1\leq x \leq2,\ 1\leq xy\leq 2};
	\end{equation*}
	the images of \mDs{1} abut along the lines $x=2^{n},\ n\in\Integers$ 
	and fill out \calV{} together with its upper and lower boundaries, the curves $xy=1$ and $xy=2$.
	Suppose the arc \Curve{} is a 
	``whisker'' for \mDs{1} (one endpoint is in\mDs{1} (say on the curve $xy=1$) and the rest is exterior to \mDs{1}.
	We can construct a diffeomorphism $\Phi$ of \Realstwo{} which is the identity at all points at distance \epsgo{}
	or more from \Curve{} and which takes \mDs{1} to the union of itself and a neighborhood 
	of \Curve.  
	Let
	\begin{equation*}
		\stmD=\interior\Phiof{\mDs{1}}.
	\end{equation*}
	 
	\begin{figure}[htbp]
		\begin{center}
			\begin{pspicture}(1,-0.5)(7,3.6)
				\psset{xunit=3, yunit=1.3}
				\psline(1,0.5)(1,2)\psline[linestyle=dashed](1,2)(1,3.2)
				\psplot{1}{2}{2 x div}
				\psline(2,0.25)(2,1)\psline[linestyle=dashed](2,1)(2,3.25)
				\psplot{1}{2}{0.5 x div}
				\pscurve[linewidth=1.5pt](1.3,1.52)(1.5,2.5)(1.55,2.7)(1.6,2.5)(1.65,2.2)(1.7,2.5)(1.75,3.0)(1.8,2.5)
				\pscurve[linewidth=.6pt](1.2,1.66)(1.253, 1.7)(1.3,1.8)(1.35,2.0)(1.4,2.2)(1.45,2.5)%
				(1.55,2.9)(1.62,2.6)(1.65,2.4)(1.68,2.6)(1.69,3)(1.75,3.2)(1.831,2.8)(1.83,2.5)(1.8,2.3)%
				(1.77,2.5)(1.75,2.78)(1.73,2.5)(1.65,2.1)(1.58,2.2)(1.55,2.55)(1.53,2.4)(1.5,2.2)(1.45,2.0)(1.4,1.7)(1.45,1.38)
				
				\rput(1.5,1){\mDs{1}}
				\uput[d](1,0){$x=1$}
				\uput[d](2,0){$x=2$}
				\uput[l](1,2){$xy=2$}
				\uput[l](1,0.5){$xy=1$}
				\rput(1.2,2.8){\Phiof{\mDs{1}}}\psline{->}(1.32,2.6)(1.4,2.3)\psline{->}(1.2,2.6)(1.1,1.9)
				\rput(1.25,0.7){\Curve}\psline{->}(1.25,0.98)(1.3,1.4)
		
	\end{pspicture}
		\caption{\stmD}
		\label{fig:stmD}
		\end{center}
	\end{figure}

	 If we start with a whisker \Curve{} contained in the band $\opint{1}{2}\times\Reals$, then we can make sure that
	 the part of \stmD{} outside \mDs{1} is also in this band.  Then a new open disc invariant under $T$ is the set
	 \begin{equation*}
	 	\calVC\eqdef\bigcup_{n\in\Integers}T^{n}\left(\stmD\right)
	 \end{equation*}
	 and $\calVC\cap\left(\clint{1}{2}\times\Reals\right)$ is a fundamental domain for the restriction of $T$ to \calVC.
	 
	 We can then use the diffeomorphism $\Phi$ to ``push forward'' the function $g$ in \refer{eqn}{metric}
	 to \stmD{}, and then use $T$ to define it to be invariant on the rest of \calVC.  It will be complete 
	 and Anosov by the
	 same arguments as we used before.
	 
	 Suppose that part of \Curve{} outside \mDs{1} is the graph of a function $\psi$ 
	 defined on $\clint{a}{b}\subset\opint{1}{2}$
	 with exactly $N$ local extrema; assume these occur at the points 
	 \begin{equation*}
	 	(\xs{2k+1},\ys{2k+1}), \quad a<\xs{1}<\cdots<\xs{2N-1}<b
	\end{equation*}
	 with maximum values increasing ($\ys{2k+1}<\ys{2k+5}$ for $k$ odd) and minimum values decreasing
	 ($\ys{2k+1}>\ys{2k+5}$ for $k$ even). Finally, assume some intermediate value 
	 $\ys{0}=\ys{2}=...=\ys{2N}=c$ occurs precisely at the points
	 \begin{equation*}
	 	a=\xs{0}<\xs{2}<\cdots<\xs{2N}=b.
	 \end{equation*}
	 We pick \eps{}  so that
	 $c$ does not belong to any of the closed intervals \clint{\ys{2k+1}-\eps}{\ys{2k+1}+\eps}.  This makes sure
	 that no horizontal line segment in \stmD{} can contain points near two different extrema.
	Write \ps{j} in place of $(\xs{2j},c)$.
	 
	 We \emph{claim}:
	 \begin{equation*}
	 	\access{\ps{0}}{\ps{k}}\geq2k+1.
	\end{equation*}
	To see this, note that the unstable leaf through \ps{j} is a component of the horizontal line $y=c4$ 
	which does not reach \ps{j+1}. Thus the most efficient way to get from \ps{j} to \ps{j+1} is to 
	move vertically along the stable leaf through \ps{j} until we ``clear'' the height $\ys{2j+1}\pm\eps$ (if possible)
	then move horizontally along an unstable leaf until we reach the stable leaf of \ps{j+1} and move vertically
	to reach it.  (It is possible that if we are too far to the left of the next extremum, the vertical leaf through \ps{j} 
	does not reach ``high'' enough to clear the hump, in which case we need to take additional horizontal steps
	to bring us closer to the hump).	
	Since each such step starts and ends with a vertical motion, it may be possible to ``meld'' the last arc in a given
	step with the first arc in the next, so our estimate from below counts each step as 2 arcs, plus the first one.
	
	With this we see that examples can be constructed with arbitrarily high degree of inaccessibility. But
	it is fairly easy to see that by making \Curve{} and \stmD{} sufficiently regular, we can insure that the degree is
	finite in each individual case.
	\end{proof}

\section{And Now for Something Completely Different...}\label{sec:parallel}

The examples constructed in \refer{sec}{examples} and \refer{sec}{access} give a wealth of 
Anosov diffeomorphisms in the plane (with or without a fixed point) with inequivalent Anosov structures.  
These examples are all built on the restriction of the hyperbolic linear automorphism $T:(x,y)\mapsto(2x,y/2)$
to an invariant open disc, and it is natural to ask if every Anosov structure in the plane is equivalent to one given
by such a restriction.  We shall answer this question in the negative, by constructing a foliation
in the plane which cannot be taken by a homeomorphism to a foliation of some open disc by horizontal (ord vertical) lines.

We will say that a foliation of \Realstwo{} is \deffont{quasi-parallel}
if there is homeomorphism taking \Realstwo{} onto an open disc \calU{} and taking each leaf to a component 
of the intersection of \calU{} with some foliation by parallel lines (which, by appropriate choice of the homeomorphism, 
we can take to be horizontal).  
Note that this is quite distinct from parallelizability of the 
foliation, which is the same as the existence of a global cross-section to the foliation (an embedded line which meets
each leaf of the foliation exactly once, transversally). A parallelizable foliation is homeomorphic to the horizontal 
foliation of the open square $\opint{0}{1}\times\opint{0}{1}$, but there are non-parallelizable foliations which are
quasi-parallel.

The simplest example of this is a foliation of \Realstwo{} with a single \emph{Reeb component},
for example consisting of all vertical lines $x=a$ for $\abs{a}\geq\frac{\pi}{2}$ together with the curves
$y=c+\sec x$, $-\frac{\pi}{2}<x<\frac{\pi}{2}$, $c\in\Reals$ (\refer{fig}{reeb}).

\begin{figure}[htbp]
\begin{center}
	\begin{pspicture}(-2,-2)(2,2)
		\psset{arrows=->}
		\multido{\rx=0.9+0.2}{10}{
			\psline(\rx,-2)(\rx,2)
			\psline(-\rx,2)(-\rx,-2)
			}
			\psplot{-0.7}{0.7}{1 x Pi mul 2 div COS div -0.5 add}
			\psplot{-0.75}{0.75}{1 x Pi mul 2 div COS div -1.0 add}
			\psplot{-0.8}{0.8}{1 x Pi mul 2 div COS div -1.5 add}
			\psplot{-0.8}{0.8}{1 x Pi mul 2 div COS div -2.0 add}
			\psplot{-0.8}{0.8}{1 x Pi mul 2 div COS div -2.5 add}
			\psplot{-0.8}{0.8}{1 x Pi mul 2 div COS div -3.0 add}
			
			\psline[linewidth=1.5pt](0.9,-2)(0.9,2)
			\psline[linewidth=1.5pt](-0.9,2)(-0.9,-2)
			
			\psline[linestyle=dashed](0,2)(0,-2)
			
			\rput*(0,0){(a)}
			\rput*(-2,0){(b)}
			\rput*(2,0){(c)}

	\end{pspicture}
\caption{Reeb component}
\label{fig:reeb}
\end{center}
\end{figure}
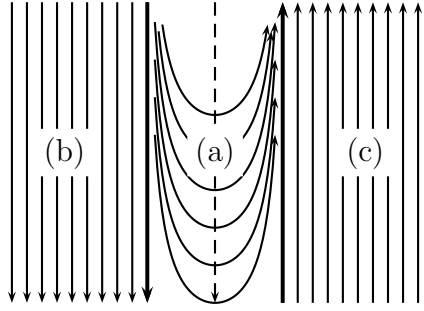

No cross-section can join the two vertical leaves at the edge of the Reeb component (the region marked (a)), 
so the foliation is not parallelizable.
However, the dashed vertical line down the middle of the Reeb component intersects every leaf 
interior to the Reeb component, so the restriction of the foliation to this open strip is parallelizable.  
By mapping this cross-section
to the open interval $\single{0}\times\opint{0}{1}$, we can clearly find a homeomorphism taking leaves of the Reeb component
to horizontal lines in the open square
$\opint{-1}{1}\times\opint{0}{1}$ and  the two edges of this component to the open intervals
$\opint{-1}{0}\times\single{0}$ and $\opint{0}{1}\times\single{0}$. We can then extend this homeomorphism so as to take
the regions marked (b) and (c) (each of which is individually parallelizable) into open triangles abutting these
two segments (\refer{fig}{Ureeb}).

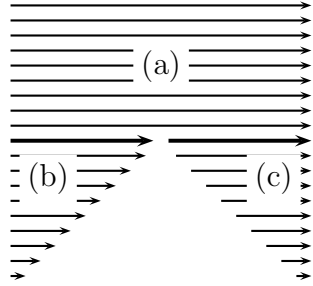
\begin{figure}[htbp]
\begin{center}
	\begin{pspicture}(-2,-2)(2,2)
		\psset{arrows=->}
		\multido{\rx=0.2+0.2}{9}{
			\psline(-2,\rx)(2,\rx)
			\psline(-2,-\rx)(!\rx\space  neg \rx\space neg)
			\psline(!\rx\space  \rx\space neg)(2,-\rx)
			}
			\psline[linewidth=1.5pt](-2,0)(-0.1,0)
			\psline[linewidth=1.5pt](0.1,0)(2,0)
			
			\rput*(0,1){(a)}
			\rput*(-1.5,-0.5){(b)}
			\rput*(1.5,-0.5){(c)}

	\end{pspicture}
\caption{Quasi-parallelization of Figure \ref{fig:reeb} }
\label{fig:Ureeb}
\end{center}
\end{figure}

To understand the situation further, it is useful to take advantage of the fact that the leaves of 
any foliation of the plane can be viewed as the orbit lines of some fixedpoint-free flow $\phi^{t}$ on \Realstwo
(\cite[Corollary to Thm. 42, p. 185]{Kaplan1}).
This allows us to introduce the idea of \emph{prolongational limit sets}:
we say that a point $q\in\Realstwo$ is a \deffont{forward prolongational limit} 
\resp{\deffont{backward prolongational limit}} of $p\in\Realstwo$ if there is a sequence of points $\ps{i}\to p$ 
and times $\ts{i}\to\infty$ \resp{$\ts{i}\to-\infty$} such that $\qs{i}=\phi^{\ts{i}}(\ps{i})\to q$.  
The set of all forward \resp{backward} prolongational limit points of $p$ is denoted $\boldsymbol{\prolong_{+}(p)}$
\resp{$\boldsymbol{\prolong_{-}(p)}$}.  In \refer{fig}{reeb3}, for each point $p$ on the left edge of either Reeb component,
\prolim{+}{p} consists of the right edge of the same Reeb component. 
\footnote{Note that the definition of non-wandering point could be written $p\in\prolim{+}{p}$.}
The existence of a nonempty prolongational limit set is the obstacle to parallelizability of a flow in the plane.

Prolongation allows us to make a subtle distinction which directly affects quasi-parallelizability.  
Consider the situation of two Reeb components, with the interior leaves in each curling up, as in \refer{fig}{reeb3},
and separated by a single orbit.

\begin{figure}[htbp]
\begin{center}
	\begin{pspicture}(-2,-2)(2,2.5)
		\psset{arrows=->}

			\psplot{-.85}{-0.15}{1 x 1 add Pi mul SIN div -0.5 add}
			\psplot{-0.88}{-0.17}{1 x 1 add Pi mul SIN div -1.0 add}
			\psplot{-0.9}{-0.1}{1 x 1 add Pi mul SIN div -1.5 add}
			\psplot{-0.9}{-0.1}{1 x 1 add Pi mul SIN div -2.0 add}
			\psplot{-0.9}{-0.1}{1 x 1 add Pi mul SIN div -2.5 add}
			\psplot{-0.9}{-0.1}{1 x 1 add Pi mul SIN div -3.0 add}
			
			\psline[linewidth=1.5pt](1,2)(1,-2)

			\psline[linewidth=1.5pt](-1,2)(-1,-2)
			
			\psline[linewidth=1.5pt](0,-2)(0,2)

	\psset{arrows=<-}
			\psplot{0.15}{0.85}{1 x 1 add Pi mul SIN div -0.5 sub neg}
			\psplot{0.17}{0.88}{1 x 1 add Pi mul SIN div -1.0 sub neg}
			\psplot{0.1}{0.9}{1 x 1 add Pi mul SIN div -1.5 sub neg}
			\psplot{0.1}{0.9}{1 x 1 add Pi mul SIN div -2.0 sub neg}
			\psplot{0.1}{0.9}{1 x 1 add Pi mul SIN div -2.5 sub neg}
			\psplot{0.1}{0.9}{1 x 1 add Pi mul SIN div -3.0 sub neg}

			\rput*(-0.5,0){(a)}
			\uput[u](0,2){(c)}
			\rput*(0.5,0){(b)}

	\end{pspicture}

\caption{Two similarly-oriented Reeb components separated by one leaf}
\label{fig:reeb3}
\end{center}
\end{figure}
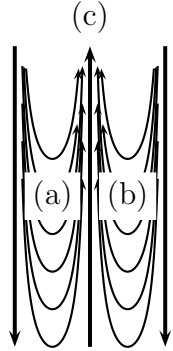

We claim this cannot be part of a quasi-parallel foliation.
To see this, note that in a quasi-parallelized picture, the horizontal lines must all be oriented in the same direction,
which we have taken to be left-to-right.  Consider the orbit (c) separating the two Reeb components, and suppose 
it maps to the open interval $I=\opint{\alpha}{\beta}$, which we can take on the \axis{x}.  Since the orbit on the
left edge of (a) is the \emph{backward} prolongational limit of these points, the orbits in (a) map to line segments extending to 
the left of $\alpha$, and the separating orbit maps to a segment \Js{1} of the real line to the left of $I$.
Furthermore, since orbits in (a) see (c) on their \emph{right} side, the image of (a) is in the upper half plane and bounded
below by a part of the axis that spans the gap between $I$ and \Js{1}.  However, (c) \emph{also} has
the right edge of (b) in its backward prolongational limit, so the orbits of (b) must also extend to the left of the
image of (c), the right edge of (b) must also map to a segment \Js{2} of the axis \emph{also} to the left of $\alpha$, 
and the image of (b) must be in the upper half plane and be bounded below by a segment of the axis which 
spans the gap between $I$ and \Js{2}.  The shorter of the two gaps is spanned by the lower edge of
the image of both regions (a) and (b), and both lie in the upper half plane.  It follows that they must intersect,
a contradiction to the fact that these are images under a homeomorphism of the whole plane into the plane.

To construct an example of an Anosov structure on the plane which is not equivalent to a restriction 
of the hyperbolic linear automorphism to an invariant disc, 
it suffices to construct an example for which one of the two foliations exhibits such a configuration.

Our example is an adaptation of Warren White's example in \cite{White} of an Anosov structure for the translation
$(x,y)\mapsto(x+2,y)$.  The basis of his example is to construct a smooth frame field (a pair of orthonormal vectors at
each point) $(\vstable(x,y),\vunstable(x,y))$ which is invariant under
all vertical translations), but rotates as the point moves horizontally
in such a way that it is invariant under a horizontal translation by one unit%
\footnote{For White's example, the unit is 2, but we will build one using unit 1.}
He also makes sure that there is a nontrivial interval of $x$-values for which \vstable{} is horizontal and another for which
\vunstable{} is horizontal.  Then given $0<\lambda<1$, the Riemann metric for which 
the inner product of a pair of vectors \vu{} and \vv{}
at $(x,y)$ is defined in terms of the Euclidean inner product by
\begin{equation*}
	\muprod{\vu}{\vv}=\lam^{2x}(\vu\cdot\vstable)(\vv\cdot\vstable)+\lam^{-2x}(\vu\cdot\vunstable)(\vv\cdot\vunstable)
\end{equation*}
gives an Anosov structure for the horizontal translation by one unit.  The hyperbolicity  of this metric is clear;
we sketch the proof that it is complete, following his argument in \cite{White}.

Let $I$ \resp{$J$} be an interval such that \vstable{} \resp{\vunstable} is horizontal at all points in the band
$\As{0}=I\times\Reals$ \resp{$\Bs{0}=J\times\Reals$}.  Let us assume that $I$ and $J$ are both contained in \opint{0}{1}
(this will simplify some notation but not substantially alter the argument) and for \inZ{n} let \As{n} \resp{\Bs{n}}
be the translate of \As{0} \resp{\Bs{0}} in $\opint{n}{n+1}\times\Reals$.  We \emph{claim} that for $n<0$ \resp{$n\geq0$} 
the width in the new metric of the band \As{n} \resp{\Bs{n}} exceeds the Euclidean length of $I$ \resp{$J$}.
If \vgamof{t} is a smooth curve connecting the left edge of the appropriate band to the right edge 
(and which we can assume to be contained in the closure of this band) then its speed in the new metric 
is at least $\max\{\lam^{x}\abs{\vgampof{t}\cdot\vstable},\lam^{-x}\abs{\vgampof{t}\cdot\vunstable}\}$,
where \abs{\vv} is the Euclidean length of \vv{}.  For a point in \As{n}, the first of these is $\lam^{x}\abs{\xpof{t}}$,
which  (since $0<\lam<1$) for $n<0$  exceeds the (Euclidean) horizontal speed, 
while in \Bs{n} and $n\geq0$ the second one
exceeds the (Euclidean) horizontal speed.  Integrating the (new) speed, we see that any curve crossing \As{n}, $n<0$
\resp{\Bs{n}, $n\geq{0}$} has length at least equal to the Euclidean width of this band.  
In particular, the width of the band $\opint{n}{n+1}\times\Reals$ in the new metric is at least equal to the (Euclidean)
length of the shorter of $I$ and $J$.
Thus, any sequence  of points which is Cauchy (hence bounded) in the new metric stays within a closed
finite vertical band $\clint{-n}{n}\times\Reals$.  Since the new metric is invariant under vertical  translations
and \clint{-n}{n} is compact, we can find uniform bounds on the distortion of lengths of vectors--for some $\Cs{n}>1$,
every vector \vv{} at a point of  $\clint{-n}{n}\times\Reals$ has new length between $1/\Cs{n}$ times its Euclidean length
and \Cs{n} times this length.  This implies analogous estimates on the ratio between the new and Euclidean distance
between two points in $\clint{-n}{n}\times\Reals$.  In particular it says that a sequence in this band converges in the 
new metric if and only if it converges in the Euclidean metric, and proves completeness.

For our version of this construction, start with a smooth function \tof{x} (\refer{fig}{tofx}) satisfying
\begin{itemize}
	\item $\tof{x}=\begin{cases}
				      & 0\text{ on } \clint{0}{0.1}, \\
				      &\frac{\pi}{2} \text{ on }, \clint{0.2}{0.4}\\
				      & \pi \text{ at } x=0.5\text{ (\emph{only})}, \\
				      & \frac{3\pi}{2} \text{ on } \clint{0.6}{0.8}, \\
				      & 2\pi \text{ on } \clint{0.9}{1.0} \\
				\end{cases}$
	\item \tof{x} is strictly increasing on each of the intervals \clint{01.}{0.2}, 
	\clint{0.4}{0.6}, and \clint{0.8}{0.9} 
	\item $\tof{x+1}=\tof{x}+2\pi$ (so $\tof{x+n}=\tof{x}+2n\pi$) for all integers $n$.
\end{itemize}

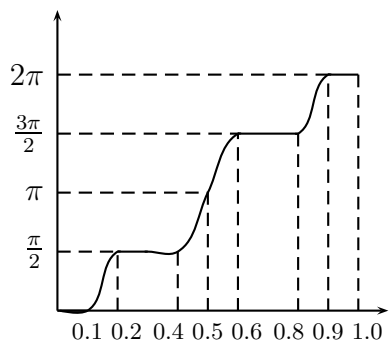
\begin{figure}[htbp]
\begin{center}
	\begin{pspicture}(0,-0.5)(1,4)
	\psset{xunit=4}
		\psline{->}(0,0)(1.1,0)		
		\psline{->}(0,0)(0,4)		
		
		\psline(0,0)(0.1,0)
		\pscurve(0.009,0)(0.1,0)(0.2,0.785)(0.21,0.785)
		\psline(0.2,0.785)(0.3,0.785)
		\pscurve(0.29,0.785)(0.3,0.785)(0.4,0.785)(0.5,1.57)(0.6,2.355)(0.61,2.355)
		\psline(0.6,2.355)(0.7,2.355)(0.8,2.355)
		\pscurve(0.79,2.355)(0.8,2.355)(0.9,3.14)(0.91,3.14)
		\psline(0.9,3.14)(1.0,3.14)
		
		\psline[linestyle=dashed](0,0.785)(0.2,0.785)(0.2,0)
		\psline[linestyle=dashed](0.4,0.785)(0.4,0)
		\psline[linestyle=dashed](0,1.57)(0.5,1.57)(0.5,0)
		\psline[linestyle=dashed](0,2.355)(0.6,2.355)(0.6,0)
		\psline[linestyle=dashed](0.8,2.355)(0.8,0)
		\psline[linestyle=dashed](0,3.14)(0.9,3.14)(0.9,0)
		\psline[linestyle=dashed](1.0,3.14)(1.0,0)
		
		\uput[l](0,0.785){$\frac{\pi}{2}$}
		\uput[l](0,1.57){$\pi$}
		\uput[l](0,2.355){$\frac{3\pi}{2}$}
		\uput[l](0,3.14){$2\pi$}
		
		\uput[d](0.1,0){\scriptsize$0.1$}
		\uput[d](0.23,0){\scriptsize$0.2$}
		\uput[d](0.37,0){\scriptsize$0.4$}
		\uput[d](0.5,0){\scriptsize$0.5$}
		\uput[d](0.63,0){\scriptsize$0.6$}
		\uput[d](0.77,0){\scriptsize$0.8$}
		\uput[d](0.9,0){\scriptsize$0.9$}
		\uput[d](1.03,0){\scriptsize$1.0$}
		
	\end{pspicture}
\caption{The function \tof{x}}
\label{fig:tofx}
\end{center}
\end{figure}

We then define the frame by
\begin{align*}
	\vstable(x,y)&=(\cos\tof{x},\sin\tof{x})\\
	\vunstable(x,y)&=(\cos\left(\tof{x}+\frac{\pi}{2}\right),\sin\left(\tof{x}+\frac{\pi}{2}\right)).
\end{align*}

In \refer{fig}{frame}, we have sketched the typical orientation of the frame in each of the intervals of definition
of \tof{x} (\vstable{} is light, \vunstable{} is dark ), as well as a typical leaf of the unstable foliation. 

\begin{figure}[htbp]
\begin{center}
	\begin{pspicture}(0,0)(10,4)
		\psline[linestyle=dashed](2,1)(2,4)
		\psline[linestyle=dashed](4,1)(4,4)
		\psline[linewidth=1.5pt, arrows=->](5,4)(5,1)
		\psline[linestyle=dashed](6,1)(6,4)
		\psline[linestyle=dashed](8,1)(8,4)
		
		\rput(0.5,0){\color{lightgray}{\psline[arrows=->](0,0)(0.4,0)} \psline[linewidth=1.5pt, arrows=->](0,0)(0,0.4)}
		\rput{45}(1.5,0){\color{lightgray}{\psline[arrows=->](0,0)(0.4,0)} \psline[linewidth=1.5pt, arrows=->](0,0)(0,0.4)}
		\rput{90}(3.0,0){\color{lightgray}{\psline[arrows=->](0,0)(0.4,0)} \psline[linewidth=1.5pt, arrows=->](0,0)(0,0.4)}
		\rput{135}(4.5,0){\color{lightgray}{\psline[arrows=->](0,0)(0.4,0)} \psline[linewidth=1.5pt, arrows=->](0,0)(0,0.4)}
		\rput{180}(5,0){\color{lightgray}{\psline[arrows=->](0,0)(0.4,0)} \psline[linewidth=1.5pt, arrows=->](0,0)(0,0.4)}
		\rput{225}(5.5,0){\color{lightgray}{\psline[arrows=->](0,0)(0.4,0)} \psline[linewidth=1.5pt, arrows=->](0,0)(0,0.4)}
		\rput{270}(7.0,0){\color{lightgray}{\psline[arrows=->](0,0)(0.4,0)} \psline[linewidth=1.5pt, arrows=->](0,0)(0,0.4)}
		\rput{315}(8.5,0){\color{lightgray}{\psline[arrows=->](0,0)(0.4,0)} \psline[linewidth=1.5pt, arrows=->](0,0)(0,0.4)}
		\rput(9.5,0){\color{lightgray}{\psline[arrows=->](0,0)(0.4,0)} \psline[linewidth=1.5pt, arrows=->](0,0)(0,0.4)}
		
		\psBezier5[linewidth=1.5pt, arrows=->](1.1,4.0)(1.1,3.0)(1.5,1.5)(1.6,1)(1.8,1)(2,1)
		\psline[linewidth=1.5pt](2,1)(4,1)
		\psBezier5[linewidth=1.5pt](4,1)(4.2,1)(4.4,1)(4.5,1.5)(4.9,3)(4.9,4)
		
		\psBezier5[linewidth=1.5pt](5.1,4.0)(5.1,3.0)(5.5,1.5)(5.6,1)(5.8,1)(6,1)
		\psline[linewidth=1.5pt, arrows=->](6,1)(8,1)
		\psBezier5[linewidth=1.5pt, arrows=<-](8,1)(8.2,1)(8.4,1)(8.5,1.5)(8.9,3)(8.9,4)
		
		\multido{\rx=0+.2}{6}{\psline[linewidth=1.5pt, arrows=->](\rx,1)(\rx,4)}
		\multido{\rx=9.0+.2}{6}{\psline[linewidth=1.5pt, arrows=->](\rx,1)(\rx,4)}
		
		\multido{\ry=2.2+0.2}{9}{\color{lightgray}{\psline{->}(\ry,1)(\ry,4)}}
		\multido{\rz=6.2+0.2}{9}{\color{lightgray}{\psline{<-}(\rz,1)(\rz,4)}}
		
		\color{lightgray}{
			\psline(0,1)(1,1)
			\psBezier5{<-}(1,1)(1.2,1)(1.4,1)(1.5,1.5)(1.9,3)(1.9,4)
			
			\psBezier8{->}(4.1,4)(4.1,2)(4.5,1.5)(4.9,1)(5.0,1)(5.1,1)(5.5,1.5)(5.9,2)(5.9,4)
			
			\psBezier5(8.1,4)(8.1,3)(8.5,1.5)(8.6,1)(8.8,1)(9,1)
			\psline{->}(9.0,1)(10.0,1)
			}
		
	\end{pspicture}
\caption{$\vec{e_{s}}$, $\vec{e_{u}}$, and the unstable foliation in our example}
\label{fig:frame}
\end{center}
\end{figure}
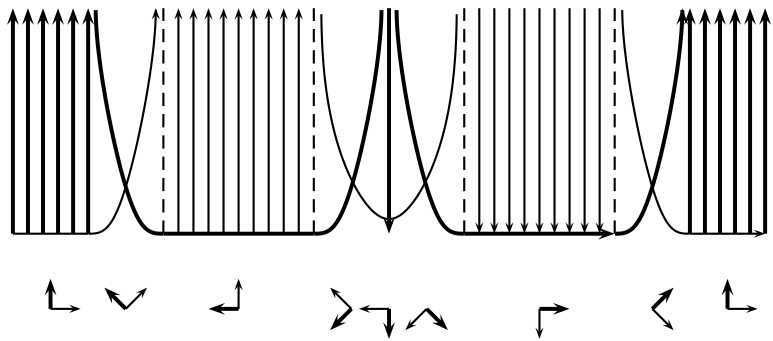

It is clear from \refer{fig}{frame} that the unstable foliation in this example is not quasi-parallel, which means this
particular Anosov structure (for the translation $(x,y)\mapsto(x+1,y)$) is not equivalent to a structure coming from
a restriction of a linear hyperbolic automorphism to an invariant open disc that excludes the origin.  
Thus, despite the variety they exhibit,
our examples in \refer{sec}{examples} and \refer{sec}{access} cannot serve as models for all Anosov structures
in the plane:

\begin{theorem}
	There exist Anosov structures for a parallel translation in the plane possessing at least one non-quasi-parallel 
	foliation, and therefore not equivalent to the restriction of a linear hyperbolic automorphism to an invariant disc
	not containing the fixed point at the origin.
\end{theorem}

\bibliography{noncompact}
\bibliographystyle{amsplain}

\end{document}